\documentclass[11pt]{amsart}

\usepackage[T1]{fontenc}
\usepackage[utf8]{inputenc}
\usepackage{lmodern}
\usepackage{amsmath,amssymb,amsthm,mathtools}
\usepackage{microtype}
\usepackage[colorlinks=true,linkcolor=blue,citecolor=blue,urlcolor=blue]{hyperref}
\hypersetup{
  pdftitle={Height Rigidity for Entire Functions},
  pdfauthor={Diego Marques},
  pdfsubject={Arithmetic entire functions and bounded-height algebraic values},
  pdfkeywords={entire function, algebraic value, absolute height, rational point, globally subanalytic set, Pila counting theorem}
}

\newtheorem{theorem}{Theorem}[section]
\newtheorem{corollary}[theorem]{Corollary}

\newtheorem{lemma}[theorem]{Lemma}

\theoremstyle{remark}
\newtheorem{remark}[theorem]{Remark}

\newcommand{\Q}{\mathbb{Q}}
\newcommand{\R}{\mathbb{R}}
\newcommand{\C}{\mathbb{C}}
\newcommand{\Z}{\mathbb{Z}}

\newcommand{\Qbar}{\overline{\mathbb{Q}}}
\newcommand{\Aatmost}[1]{\mathcal{A}_{\leq #1}}
\newcommand{\Aexact}[1]{\mathcal{A}_{#1}}
\DeclareMathOperator{\den}{den}
\DeclareMathOperator{\RePart}{Re}
\DeclareMathOperator{\ImPart}{Im}

\title[Height rigidity for entire functions]{Height Rigidity for Entire Functions}

\author{Diego Marques}
\address{Departamento de Matem\'atica, Universidade de Bras\'ilia, Bras\'ilia, DF, Brazil}
\email{diego@mat.unb.br}

\subjclass[2020]{Primary 11J81; Secondary 03C64, 11G50, 30D20}
\keywords{Entire function, algebraic value, absolute height, rational point, globally subanalytic set, Pila's counting theorem}

\begin{document}

\begin{abstract}
We prove that algebraic values of bounded degree and polynomially bounded height are sparse on rational translates of the graph of a transcendental entire function. More precisely, for fixed $\theta\in\Qbar\cap\R$, $D\geq1$, $t>0$, and for every $\varepsilon>0$, only $O(Q^{\varepsilon})$ rationals $r\in[0,1]$ of height at most $Q$ can satisfy simultaneously $[\Q(f(\theta+r)):\Q]\leq D$ and $H(f(\theta+r))\ll H(r)^t$. Consequently, if these bounds hold for every rational $r$ of sufficiently large height, then $f\in\Qbar[z]$ and $\deg f\leq t$. As applications, we obtain rigidity results for entire functions taking rational or bounded-degree algebraic values with polynomially controlled arithmetic height. In particular, this excludes the polynomial-denominator scenario that arises naturally in connection with Mahler's problem on Liouville numbers. The proof combines Pila's bounded-degree counting theorem with standard height estimates and a simple geometric analysis of transcendental entire graphs.
\end{abstract}

\maketitle

\section{Introduction}

The arithmetic behavior of transcendental analytic functions has a long and somewhat paradoxical history. Classical interpolation constructions of Weierstrass \cite{Weierstrass1923} and St\"ackel \cite{Stackel}, discussed for instance by Waldschmidt \cite{Waldschmidt2003}, show that transcendental entire functions may take prescribed arithmetic values on large countable sets. Van der Poorten \cite{vanderPoorten1968} subsequently constructed transcendental entire functions preserving every algebraic number field. More recently, Marques and Moreira \cite{MarquesMoreira2017} proved the existence of uncountably many transcendental entire functions with rational Taylor coefficients for which both the image and the full preimage of $\Qbar$ consist of algebraic numbers, thereby answering a question of Mahler. These results emphasize that algebraicity of values alone imposes surprisingly little rigidity.

A different picture emerges once one asks for quantitative control of arithmetic complexity. For a rational number $x=a/b$ in lowest terms, set
\[
H(x)=\max\{|a|,b\},
\qquad
\den(x)=b.
\]
For a general algebraic number $\alpha$, we write $H(\alpha)$ for its absolute multiplicative Weil height and $h(\alpha)=\log H(\alpha)$ for the logarithmic height. This naturally raises a quantitative rigidity question: how many algebraic points of controlled degree and height can lie on the graph of a transcendental entire function when the inputs themselves have controlled height?

The systematic study of rational and algebraic points on transcendental graphs originates in the determinant method of Bombieri and Pila \cite{BombieriPila1989} and in the work of Pila \cite{Pila1991} on smooth transcendental graphs. It was transformed by Pila and Wilkie \cite{PilaWilkie2006} through their counting theorem for sets definable in o-minimal structures. Of particular relevance here is Pila's bounded-degree refinement \cite{Pila2009}; in the globally subanalytic case used below, it gives a subpolynomial estimate for algebraic points of fixed degree and bounded height outside the positive-dimensional semialgebraic part. Related quantitative investigations of algebraic values of analytic functions were developed by Surroca \cite{Surroca2002,Surroca2006}, Boxall and Jones \cite{BoxallJones2015a,BoxallJones2015b}, and Comte and Yomdin \cite{ComteYomdin2018}.

There is also a substantial literature on arithmetic entire functions, integer-valued functions, and rational values of analytic functions; see Pila and Rodr\'iguez Villegas \cite{PilaRodriguezVillegas1999}, Lombardo \cite{Lombardo2017}, and Waldschmidt \cite{Waldschmidt2020}.

A direction especially close to the present work originates in the arithmetic of Liouville numbers. Recall that a real number $\xi$ is called a \emph{Liouville number} if, for every integer $n\geq1$, there exist integers $p$ and $q\geq2$ such that
\[
0<\left|\xi-\frac{p}{q}\right|<\frac{1}{q^n}.
\]
Equivalently, Liouville numbers are precisely the real numbers with infinite irrationality exponent. Their exceptional rational approximation properties make the behavior of analytic functions on this set particularly delicate. A classical theorem of Maillet \cite{Maillet1906} asserts that every nonconstant rational function with rational coefficients maps Liouville numbers to Liouville numbers. Motivated by this phenomenon, Mahler \cite{Mahler1984} asked whether there exists a transcendental entire function with the same property. Marques and Moreira \cite{MarquesMoreira2015} subsequently connected this question with the arithmetic complexity of rational values of entire functions. They constructed uncountably many transcendental entire functions $f$ satisfying $f(\Q)\subseteq\Q$ and
\[
\den\!\left(f\!\left(\frac pq\right)\right)<q^{8q^2},
\]
for every rational $p/q$ with $q>1$. More importantly for the present paper, their argument shows that a positive answer to Mahler's problem would follow from the existence of a transcendental entire function $f$ with $f(\Q)\subseteq\Q$ for which
\[
\den\!\left(f\!\left(\frac pq\right)\right)=O(q^t)
\]
for some fixed $t>0$. Thus polynomial control of the denominators of rational values emerges naturally as a quantitative threshold between arithmetic interpolation and analytic rigidity.

This existence problem has since been attacked under progressively weaker hypotheses. Marques, Ramirez, and Silva \cite{MarquesRamirezSilva2016} proved that no transcendental lacunary entire function $f\in\Q[[z]]$ mapping $\Q$ into itself can satisfy $\den(f(p/q))=o(q)$; in particular, this rules out polynomial bounds $O(q^t)$ for $0<t<1$ in the rational-coefficient setting. Lelis and Marques \cite{LelisMarques2020} settled the borderline exponent $t=1$ without any assumption on the Taylor coefficients. More recently, Lelis, Marques, Moreira, and Trojovsk\'y \cite{LelisMarquesMoreiraTrojovsky2024} proved nonexistence for every polynomial exponent $t>0$ when the Taylor coefficients are rational. In sharp contrast, Lelis, Moreira, and Silva \cite{LelisMoreiraSilva2026} showed that polynomial denominator growth becomes possible in finite smoothness: for every positive integer $k$ and every $t>2k$, there are uncountably many $C^k$-functions mapping $\Q$ into itself with denominator growth $O(q^t)$, and such functions can in fact be chosen to preserve the set of Liouville numbers. They also recall the classical consequence of Pila's method that rational points of bounded height on a transcendental analytic graph are subpolynomially sparse, and conclude by asking about analogous arithmetic constraints for algebraic inputs of bounded degree, with polynomial control of the height and possibly of the degree of the corresponding values.

In the purely rational setting, Pila's rational-point estimate already suggests a direct obstruction to polynomial denominator growth. Indeed, on a compact interval an entire function is bounded, so that
\[
\den\!\left(f\!\left(\frac pq\right)\right)\ll q^t
\]
together with $f(p/q)\in\Q$ implies a polynomial bound for the height of $f(p/q)$. The quadratic abundance of rationals of height at most $Q$ is then incompatible with a subpolynomial counting estimate for a transcendental graph. The point of the present work is that this phenomenon persists in a substantially broader arithmetic setting: rational values are replaced by algebraic values of uniformly bounded degree, denominator bounds by absolute Weil-height bounds, and the resulting rigidity determines not only algebraicity of the function but also its degree.

More precisely, Corollary~\ref{cor:rational} shows that an entire function satisfying $f(\Q)\subseteq\Q$ and
\[
\den\!\left(f\!\left(\frac pq\right)\right)\ll q^t
\]
for some $t>0$ must belong to $\Q[z]$, with degree at most $t$. This rational case is a specialization of the broader height-rigidity principle established in Theorem~\ref{thm:rigidity}: algebraic values of uniformly bounded degree cannot have polynomially bounded Weil height along a rational translate unless the function is a polynomial. In particular, in the entire setting, Corollary~\ref{cor:degree} gives a negative answer to the bounded-degree polynomial-height direction suggested by Lelis, Moreira, and Silva \cite{LelisMoreiraSilva2026}: a transcendental entire function cannot map algebraic numbers of fixed or bounded degree to algebraic numbers of uniformly bounded degree while satisfying
\[
H(f(\alpha))\ll H(\alpha)^t.
\]
The obstruction is therefore genuinely height-theoretic and bounded-degree in nature, rather than a peculiarity of rational denominators.

The aim of this paper is to isolate this unified height-rigidity principle and to quantify the sparsity phenomenon from which it arises.

For an entire function $f$, a real algebraic number $\theta$, an integer $D\geq1$, and real numbers $A,t>0$, define
\[
\mathcal N_f(Q;\theta,D,t,A)
=
\#\left\{
\begin{array}{c}
r\in\Q\cap[0,1]: H(r)\leq Q,\\[2mm]
[\Q(f(\theta+r)):\Q]\leq D,\\[1mm]
H(f(\theta+r))\leq A H(r)^t
\end{array}
\right\}.
\]

Our first result is a quantitative sparsity theorem.

\begin{theorem}[Arithmetic sparsity]\label{thm:sparsity}
Let $f:\C\to\C$ be a transcendental entire function, let $\theta\in\Qbar\cap\R$, let $D\geq1$ be an integer, and let $A,t>0$. Then, for every $\varepsilon>0$,
\[
\mathcal N_f(Q;\theta,D,t,A)
\ll_{f,\theta,D,t,A,\varepsilon} Q^{\varepsilon}
\]
for all $Q\geq2$.
\end{theorem}

Thus a transcendental entire graph cannot contain a polynomially dense family of algebraic points whose coordinate degrees are uniformly bounded and whose output heights grow only polynomially in the input height. Comparing this sparsity with the quadratic growth of rational points yields the main structural result.

\begin{theorem}[Height rigidity]\label{thm:rigidity}
Let $f:\C\to\C$ be an entire function and let $\theta\in\Qbar\cap\R$. Suppose that there exist an integer $D\geq1$, real numbers $A,t>0$, and $H_0\geq1$ such that, for every $r\in\Q\cap[0,1]$ with $H(r)\geq H_0$,
\begin{equation}\label{eq:intro-assumptions}
[\Q(f(\theta+r)):\Q]\leq D
\qquad\text{and}\qquad
H(f(\theta+r))\leq A H(r)^t.
\end{equation}
Then
\[
f\in\Qbar[z]
\qquad\text{and}\qquad
\deg f\leq t.
\]
\end{theorem}

The exponent in Theorem~\ref{thm:rigidity} is sharp: a polynomial $P$ of degree $m$ satisfies
\[
h(P(\theta+r))=m h(r)+O_{P,\theta}(1)
\]
along a rational translate. Thus the conclusion $\deg f\leq t$ cannot be improved. Two immediate applications recover the rational denominator problem and settle the corresponding bounded-degree formulation.

\begin{corollary}[Polynomial denominator growth]\label{cor:rational}
Let $f:\C\to\C$ be an entire function such that $f(\Q)\subseteq\Q$. Assume that there exist $A,t>0$ and $q_0\geq1$ such that
\[
\den\!\left(f\!\left(\frac pq\right)\right)
\leq A q^t
\]
for every reduced rational number $p/q$ with $q\geq q_0$. Then
\[
f\in\Q[z]
\qquad\text{and}\qquad
\deg f\leq t.
\]
In particular, no transcendental entire function has these properties.
\end{corollary}

To formulate the bounded-degree consequences, we distinguish the two conventions that are often denoted informally by $\Q_d$:
\[
\Aatmost{d}
:=\{\alpha\in\Qbar:[\Q(\alpha):\Q]\leq d\},
\qquad
\Aexact{d}
:=\{\alpha\in\Qbar:[\Q(\alpha):\Q]=d\}.
\]

\begin{corollary}[Algebraic numbers of bounded or fixed degree]\label{cor:degree}
Let $d,D\geq1$ be integers, let $f:\C\to\C$ be an entire function, and suppose that, for some $A,t>0$ and $H_0\geq1$,
\[
H(f(\alpha))\leq A H(\alpha)^t
\]
whenever $\alpha$ belongs to the relevant source set below and $H(\alpha)\geq H_0$.

\begin{enumerate}
\item If $f(\Aatmost{d})\subseteq\Aatmost{D}$, then $f\in\Qbar[z]$ and $\deg f\leq t$.
\item If $f(\Aexact{d})\subseteq\Aatmost{D}$, then $f\in\Qbar[z]$ and $\deg f\leq t$.
\end{enumerate}

In particular, neither $\Aatmost{d}$ nor $\Aexact{d}$ can be preserved under a polynomial height bound by a transcendental entire function.
\end{corollary}

The paper is organized as follows. Section~\ref{S2} collects the required height estimates, Pila's bounded-degree counting theorem, and the geometric fact that the compact graph associated with a nonpolynomial entire function is globally subanalytic and contains no connected positive-dimensional semialgebraic subset. In Section~\ref{S3}, the resulting $O(Q^\varepsilon)$ upper bound is compared with the $\gg Q^2$ rational points of bounded height to prove the main rigidity theorem and its applications; we conclude with the sharpness of the degree bound.

\section{Preliminary results}\label{S2}

Throughout, $\Qbar$ denotes the algebraic closure of $\Q$ in $\C$, and
$\deg\alpha=[\Q(\alpha):\Q]$ for $\alpha\in\Qbar$. For nonnegative quantities
$X$ and $Y$, we write $X\ll Y$, or equivalently $X=O(Y)$, if $X\leq CY$
for some constant $C>0$. A subscript indicates the permitted dependence
of the implied constant; thus, for instance, $X\ll_{\kappa}Y$ means that
$C$ may depend on $\kappa$. We write $X\gg Y$ for $Y\ll X$ and
$X=o(Y)$ if $X/Y\to0$ in the indicated limit. Implied constants may change
from line to line, $\#E$ denotes the cardinality of a finite set $E$, and
all logarithms are natural.

\subsection{Heights and algebraic degree}

We use the absolute \textit{multiplicative Weil height} $H$ on $\Qbar$. If
$\alpha\in\Qbar$ has degree $d$ and primitive minimal polynomial
\[
P_\alpha(X)=a_d\prod_{j=1}^{d}(X-\alpha_j)\in\Z[X],
\qquad a_d>0,
\]
then
\[
H(\alpha)=
\left(a_d\prod_{j=1}^{d}\max\{1,|\alpha_j|\}\right)^{1/d}.
\]
We write $h=\log H$ and, for the zero polynomial, we adopt the convention $\deg 0=-\infty$. In particular,
$H(p/q)=\max\{|p|,q\}$ for $p/q\in\Q$ in lowest terms. We shall use the standard inequalities
\begin{equation}\label{eq:height-basic}
H(\alpha\beta)\leq H(\alpha)H(\beta),
\qquad
H(\alpha+\beta)\leq2H(\alpha)H(\beta),
\end{equation}
together with $H(\sigma(\alpha))=H(\alpha)$ for every $\Q$-embedding
$\sigma:\Qbar\hookrightarrow\C$. See \cite[Chapter~1]{BombieriGubler2006}.

\begin{lemma}\label{lem:translate-height}
Let $\theta\in\Qbar$ be fixed. For every $r\in\Q$,
\[
[\Q(\theta+r):\Q]=[\Q(\theta):\Q]
\]
and
\begin{equation}\label{eq:translate-comparison}
\frac{H(r)}{2H(\theta)}
\leq H(\theta+r)
\leq2H(\theta)H(r).
\end{equation}
Equivalently,
\[
|h(\theta+r)-h(r)|
\leq h(\theta)+\log2.
\]
\end{lemma}

\begin{proof}
The degree assertion follows immediately. The upper bound in \eqref{eq:translate-comparison} is an immediate consequence of \eqref{eq:height-basic}. Since $r=(\theta+r)-\theta$, the same inequality gives
\[
H(r)\leq2H(\theta+r)H(\theta),
\]
which is the lower bound.
\end{proof}

Because the graph of a complex-valued function on a real interval is naturally embedded in $\R^3$, we need degree and height estimates for real and imaginary parts.

\begin{lemma}\label{lem:real-imaginary}
Let $\beta\in\Qbar$ and assume that $[\Q(\beta):\Q]\leq D$. Then
\[
[\Q(\RePart\beta):\Q]\leq D^2,
\qquad
[\Q(\ImPart\beta):\Q]\leq2D^2,
\]
and
\[
H(\RePart\beta)\leq4H(\beta)^2,
\qquad
H(\ImPart\beta)\leq4H(\beta)^2.
\]
\end{lemma}

\begin{proof}
Complex conjugation preserves algebraic degree and height. Since
\[
\RePart\beta=\frac{\beta+\overline\beta}{2},
\qquad
\ImPart\beta=\frac{\beta-\overline\beta}{2i},
\]
we have
\[
\Q(\RePart\beta)\subseteq\Q(\beta,\overline\beta),
\qquad
\Q(\ImPart\beta)\subseteq\Q(\beta,\overline\beta,i).
\]
The claimed degree bounds follow because
\[
[\Q(\beta,\overline\beta):\Q]\leq D^2
\]
and $[\Q(i):\Q]=2$. Moreover, \eqref{eq:height-basic}, together with $H(1/2)=H(1/(2i))=2$, gives
\[
H(\RePart\beta)
\leq H(1/2)H(\beta+\overline\beta)
\leq4H(\beta)^2,
\]
and the same argument applies to $\ImPart\beta$.
\end{proof}

We shall also use the standard functorial height estimate for polynomials.

\begin{lemma}\label{lem:polynomial-height}
Let $P\in\Qbar[z]$ be a nonconstant polynomial of degree $m$. Then
\begin{equation}\label{eq:polynomial-height}
h(P(\alpha))=m h(\alpha)+O_P(1)
\end{equation}
for every $\alpha\in\Qbar$.
\end{lemma}

\begin{proof}
The polynomial $P$ extends, by setting $P(\infty)=\infty$, to a morphism
\[
\mathbb P^1\longrightarrow\mathbb P^1
\]
of degree $m$. The assertion is therefore the standard functorial height
estimate
\[
h(P(\alpha))=m h(\alpha)+O_P(1);
\]
see Bombieri and Gubler \cite[Chapter~1]{BombieriGubler2006}.
\end{proof}

\subsection{Pila's bounded-degree counting theorem}

For $X\subseteq\R^n$, an integer $k\geq1$, and $T\geq1$, define
\[
X(k,T):=
\left\{
(x_1,\ldots,x_n)\in X:
[\Q(x_j):\Q]\leq k,\;
H(x_j)\leq T
\text{ for }1\leq j\leq n
\right\}.
\]

Recall that a subset of $\R^n$ is \emph{semialgebraic} if it can be
described by finitely many polynomial equalities and inequalities, combined
using finite unions, intersections, and complements. We shall also use
the standard fact that the graph of a real-analytic map restricted to a
compact interval is \textit{globally subanalytic}; see van den Dries
\cite{vandenDries1998}.

\begin{theorem}[Pila]\label{thm:pila}
Let $X\subseteq\R^n$ be globally subanalytic and suppose that $X$ contains
no connected positive-dimensional semialgebraic subset. Then, for every
integer $k\geq1$ and every $\varepsilon>0$,
\[
\#X(k,T)\ll_{X,k,\varepsilon}T^\varepsilon
\]
for every $T\geq1$.
\end{theorem}

This is the only form of Pila's bounded-degree counting theorem needed
here. It follows immediately from \cite[Theorem~1.6]{Pila2009}, since under
the stated hypothesis the algebraic part occurring in the general theorem
is empty.

\subsection{Transcendental entire graphs}

For an entire function $f$ and $\theta\in\R$, put
\[
\Gamma_{f,\theta}
:=
\{(\theta+x,u(\theta+x),v(\theta+x)):0\leq x\leq1\},
\]
where, for real $x$,
\[
u(x)=\RePart f(x),
\qquad
v(x)=\ImPart f(x).
\]

\begin{lemma}\label{lem:graph}
The set $\Gamma_{f,\theta}$ is globally subanalytic. If $f$ is not a
polynomial, then $\Gamma_{f,\theta}$ contains no connected
positive-dimensional semialgebraic subset.
\end{lemma}

\begin{proof}
The map
\[
x\longmapsto (x,u(x),v(x))
\]
is real analytic on a neighborhood of the compact interval
$[\theta,\theta+1]$. Its graph $\Gamma_{f,\theta}$ is therefore globally
subanalytic.

Suppose that $\Gamma_{f,\theta}$ contains a connected positive-dimensional
semialgebraic subset $S$. Since $\Gamma_{f,\theta}$ is the graph of a
function of one real variable, $S$ has dimension one. The first-coordinate
projection
\[
\pi:S\longrightarrow\R,
\qquad
\pi(x,y,z)=x,
\]
is injective and semialgebraic. Hence $\pi(S)$ is a one-dimensional
semialgebraic subset of $\R$ and therefore contains a nonempty open interval
$J$. For every $x\in J$, the unique point of $\Gamma_{f,\theta}$ above $x$
belongs to $S$. Consequently, the restrictions $u|_J$ and $v|_J$ have
semialgebraic graphs.

Since $u|_J$ and $v|_J$ are also real analytic, they are \textit{Nash functions} and
hence algebraic over $\R(x)$; see Bochnak, Coste, and Roy
\cite[Proposition~8.1.8]{BochnakCosteRoy1998}. It follows that
$f=u+iv$ is algebraic over $\C(x)$ on $J$. After replacing $J$ by a smaller nonempty open subinterval, if necessary,
and clearing denominators, there exists a nonzero polynomial
$P\in\C[X,Y]$ such that
\[
P(x,f(x))=0
\qquad(x\in J).
\]
By the identity theorem, the entire function $z\mapsto P(z,f(z))$ vanishes identically. Hence $f$ is algebraic over $\C(z)$. Since every entire algebraic function is a polynomial, it follows that $f$ is a polynomial, a contradiction.
\end{proof}

\subsection{Rational points of bounded height in an interval}

\begin{lemma}\label{lem:farey}
For every fixed $H_0\geq1$, there are constants $c>0$ and $Q_0\geq1$ such that
\[
\#\{r\in\Q\cap[0,1]:H_0\leq H(r)\leq Q\}
\geq cQ^2
\]
for every $Q\geq Q_0$.
\end{lemma}

\begin{proof}
If $r=p/q\in(0,1)$ is written in lowest terms, then $0<p<q$ and
\[
H(r)=\max\{p,q\}=q.
\]
Consequently,
\[
\begin{aligned}
\#\{r\in\Q\cap[0,1]:H(r)\leq Q\}
&=
2+\sum_{2\leq q\leq Q}\varphi(q)\\
&=
1+\sum_{q\leq Q}\varphi(q),
\end{aligned}
\]
where the term $2$ accounts for the endpoints $0$ and $1$, and we have
used $\varphi(1)=1$. The classical summatory estimate for Euler's
totient function gives
\[
\sum_{q\leq Q}\varphi(q)
=
\frac{3}{\pi^2}Q^2+O(Q\log Q);
\]
see, for instance, \cite[Theorem~3.7]{Apostol1976}. It follows that
\[
\#\{r\in\Q\cap[0,1]:H(r)\leq Q\}
=
\frac{3}{\pi^2}Q^2+O(Q\log Q).
\]

Since $H_0$ is fixed, removing the finitely many rationals $r$ satisfying
$H(r)<H_0$ changes the preceding count by only $O_{H_0}(1)$. Therefore,
\[
\#\{r\in\Q\cap[0,1]:H_0\leq H(r)\leq Q\}
=
\frac{3}{\pi^2}Q^2+O_{H_0}(Q\log Q).
\]
Because $Q\log Q=o(Q^2)$, there exists $Q_0=Q_0(H_0)$ such that
\[
\#\{r\in\Q\cap[0,1]:H_0\leq H(r)\leq Q\}
\geq \frac{1}{\pi^2}Q^2
\]
for every $Q\geq Q_0$. Thus the assertion holds, for example, with
$c=1/\pi^2$.
\end{proof}

We are now in a position to prove the main results.

\section{Proofs and consequences}\label{S3}

\subsection{Proof of Theorem~\ref{thm:sparsity}}

Let
\[
\mathcal R(Q)
:=
\left\{
\begin{array}{c}
r\in\Q\cap[0,1]:H(r)\leq Q,\\[1mm]
[\Q(f(\theta+r)):\Q]\leq D,\\[1mm]
H(f(\theta+r))\leq A H(r)^t
\end{array}
\right\}.
\]
For $r\in\mathcal R(Q)$, write
\[
\beta_r=f(\theta+r)
\]
and associate to $r$ the point
\[
P_r
:=
(\theta+r,\RePart\beta_r,\ImPart\beta_r)
\in\Gamma_{f,\theta}.
\]
Distinct rationals $r$ give distinct points $P_r$.

Set
\[
k:=\max\{[\Q(\theta):\Q],2D^2\}
\qquad\text{and}\qquad
s:=\max\{1,2t\}.
\]
By Lemma~\ref{lem:translate-height}, the first coordinate of $P_r$ has degree at most $k$ and height at most
\[
2H(\theta)Q.
\]
By Lemma~\ref{lem:real-imaginary}, the other two coordinates have degree at most $k$ and height at most
\[
4H(\beta_r)^2
\leq4A^2H(r)^{2t}
\leq4A^2Q^{2t}.
\]
Consequently, there is a constant $K=K(\theta,A)\geq1$ such that
\[
P_r\in\Gamma_{f,\theta}(k,KQ^s)
\qquad(r\in\mathcal R(Q)).
\]

Since $f$ is transcendental entire, it is not a polynomial. By
Lemma~\ref{lem:graph}, the set $\Gamma_{f,\theta}$ is globally subanalytic
and contains no connected positive-dimensional semialgebraic subset.
Applying Theorem~\ref{thm:pila} with exponent $\varepsilon/s$, we obtain
\[
\#\mathcal R(Q)
\leq
\#\Gamma_{f,\theta}(k,KQ^s)
\ll
(KQ^s)^{\varepsilon/s}
\ll Q^{\varepsilon}.
\]
This is the desired estimate.
\qed

\subsection{Proof of Theorem~\ref{thm:rigidity}}

Suppose first that $f$ is a transcendental entire function. By Lemma~\ref{lem:farey}, the number of rationals $r\in\Q\cap[0,1]$ satisfying $H_0\leq H(r)\leq Q$ is at least $cQ^2$ for all sufficiently large $Q$. Every one of these rationals satisfies \eqref{eq:intro-assumptions}. Thus
\[
\mathcal N_f(Q;\theta,D,t,A)
\geq cQ^2.
\]
This contradicts Theorem~\ref{thm:sparsity}, for instance with $\varepsilon=1$. Hence $f$ is a polynomial. If $f=0$, the conclusion is immediate. Assume henceforth that $f\neq0$, and write
\[
f(z)=c_0+c_1z+\cdots+c_mz^m.
\]
Choose $m+1$ distinct rationals $r_0,\ldots,r_m\in[0,1]$ of height at least $H_0$. By hypothesis, the numbers
\[
x_j=\theta+r_j
\qquad\text{and}\qquad
f(x_j)
\]
are algebraic. The coefficient vector $(c_0,\ldots,c_m)$ is obtained by solving the Vandermonde system
\[
\sum_{\ell=0}^{m}c_\ell x_j^\ell=f(x_j)
\qquad(0\leq j\leq m).
\]
Since its determinant is nonzero and all entries and right-hand sides are algebraic,
\[
c_0,\ldots,c_m\in\Qbar,
\]
so $f\in\Qbar[z]$.

It remains to prove the degree bound. If $m=0$, there is nothing to show. Assume $m\geq1$. By Lemma~\ref{lem:polynomial-height}, there is a constant $C_f$ such that
\[
h(f(\alpha))
\geq m h(\alpha)-C_f
\]
for every algebraic $\alpha$. By Lemma~\ref{lem:translate-height},
\[
h(\theta+r)
\geq h(r)-h(\theta)-\log2.
\]
Therefore, for every sufficiently large rational $r\in[0,1]$,
\[
\begin{aligned}
t h(r)+\log A
&\geq h(f(\theta+r))\\
&\geq m h(\theta+r)-C_f\\
&\geq m h(r)-m(h(\theta)+\log2)-C_f.
\end{aligned}
\]
Since rationals in $[0,1]$ have arbitrarily large height, this is possible only if $m\leq t$.
\qed

\subsection{Proof of Corollary~\ref{cor:rational}}

Restrict attention to reduced rationals $x=p/q\in[0,1]$. Since $f$ is continuous on $[0,1]$, there is $M\geq1$ such that
\[
|f(x)|\leq M
\qquad(0\leq x\leq1).
\]
Write $f(p/q)=a/d$ in lowest terms. The denominator hypothesis gives
\[
d\leq Aq^t.
\]
Moreover,
\[
|a|=d\left|f\left(\frac pq\right)\right|
\leq Md,
\]
and hence
\[
H\!\left(f\!\left(\frac pq\right)\right)
\leq M A q^t
=MAH(p/q)^t.
\]
Theorem~\ref{thm:rigidity}, with $\theta=0$ and $D=1$, yields
\[
f\in\Qbar[z]
\qquad\text{and}\qquad
\deg f\leq t.
\]
If $f=0$, there is nothing further to prove. Otherwise, let $m=\deg f$. Then $f(0),f(1),\ldots,f(m)$ are rational. Inverting the corresponding rational Vandermonde matrix shows that every coefficient of $f$ is rational. Thus $f\in\Q[z]$.
\qed

\begin{remark}\label{rem:previous}
Corollary~\ref{cor:rational} contains the exponent-one result of Lelis and Marques \cite{LelisMarques2020} and removes the rational-Taylor-coefficient assumption from the polynomial-exponent rigidity theorem of Lelis, Marques, Moreira, and Trojovsk\'y \cite{LelisMarquesMoreiraTrojovsky2024}. 
\end{remark}

\subsection{Proof of Corollary~\ref{cor:degree}}

Assume first that $f(\Aatmost{d})\subseteq\Aatmost{D}$. Since every rational number belongs to $\Aatmost{d}$, the asserted degree and height bounds apply to all rational $r\in[0,1]$ of sufficiently large height. Theorem~\ref{thm:rigidity}, with $\theta=0$, gives $f\in\Qbar[z]$ and $\deg f\leq t$.

Now suppose that $f(\Aexact{d})\subseteq\Aatmost{D}$. Choose a real algebraic number $\theta$ of degree exactly $d$. For every $r\in\Q$,
\[
\Q(\theta+r)=\Q(\theta),
\]
so $\theta+r\in\Aexact{d}$. By Lemma~\ref{lem:translate-height}, $H(\theta+r)\to\infty$ as $H(r)\to\infty$ and
\[
H(\theta+r)^t
\leq (2H(\theta))^tH(r)^t.
\]
Consequently, for all rational $r\in[0,1]$ of sufficiently large height,
\[
[\Q(f(\theta+r)):\Q]\leq D
\]
and
\[
H(f(\theta+r))
\leq A(2H(\theta))^tH(r)^t.
\]
Another application of Theorem~\ref{thm:rigidity} completes the proof.
\qed

\begin{remark}[Sharpness of the exponent]\label{rem:sharpness}
The degree bound in Theorem~\ref{thm:rigidity} is optimal. Indeed, if
$P\in\Qbar[z]$ has degree $m$ and $\theta\in\Qbar$, then
\[
h(P(\theta+r))
=
m h(r)+O_{P,\theta}(1)
\qquad (r\in\Q).
\]
Moreover, the values $P(\theta+r)$ all lie in a fixed number field.
Consequently, there exist $D\geq1$ and $C>0$ such that
\[
[\Q(P(\theta+r)):\Q]\leq D,
\qquad
H(P(\theta+r))\leq C H(r)^m
\]
for every $r\in\Q$. Thus the exponent $t=m$ is attained by polynomials of
degree $m$.
\end{remark}

\section*{Statements and Declarations}

\noindent\textbf{Funding.}
This work was supported by the National Council for Scientific and Technological Development (CNPq), under Grant No.~304467/2023-5.

\medskip

\noindent\textbf{Competing interests.}
The author has no relevant financial or non-financial interests to disclose.

\medskip

\noindent\textbf{Data availability.}
No datasets were generated or analysed during the current study.


\begin{thebibliography}{99}

\bibitem{Apostol1976}
T.~M.~Apostol,
\emph{Introduction to Analytic Number Theory},
Undergraduate Texts in Mathematics,
Springer-Verlag, New York--Heidelberg, 1976.

\bibitem{BochnakCosteRoy1998}
J.~Bochnak, M.~Coste, and M.-F.~Roy,
\emph{Real Algebraic Geometry},
Ergebnisse der Mathematik und ihrer Grenzgebiete, vol.~36,
Springer-Verlag, Berlin, 1998.

\bibitem{BombieriGubler2006}
E.~Bombieri and W.~Gubler,
\emph{Heights in Diophantine Geometry},
New Mathematical Monographs, vol.~4,
Cambridge University Press, Cambridge, 2006.

\bibitem{BombieriPila1989}
E.~Bombieri and J.~Pila,
The number of integral points on arcs and ovals,
\textit{Duke Math. J.} \textbf{59} (1989), no.~2, 337--357.

\bibitem{BoxallJones2015a}
G.~Boxall and G.~O.~Jones,
Algebraic values of certain analytic functions,
\textit{Int. Math. Res. Not. IMRN} \textbf{2015} (2015), no.~4, 1141--1158.

\bibitem{BoxallJones2015b}
G.~Boxall and G.~O.~Jones,
Rational values of entire functions of finite order,
\textit{Int. Math. Res. Not. IMRN} \textbf{2015} (2015), no.~22, 12251--12264.

\bibitem{ComteYomdin2018}
G.~Comte and Y.~Yomdin,
Zeroes and rational points of analytic functions,
\textit{Ann. Inst. Fourier (Grenoble)} \textbf{68} (2018), no.~6, 2445--2476.

\bibitem{LelisMarques2020}
J.~Lelis and D.~Marques,
On transcendental entire functions mapping $\Q$ into itself,
\textit{J. Number Theory} \textbf{206} (2020), 310--319.

\bibitem{LelisMarquesMoreiraTrojovsky2024}
J.~Lelis, D.~Marques, C.~G.~Moreira, and P.~Trojovsk\'y,
A note on transcendental analytic functions with rational coefficients mapping $\Q$ into itself,
\textit{Proc. Japan Acad. Ser. A Math. Sci.} \textbf{100} (2024), no.~8, 43--45.

\bibitem{LelisMoreiraSilva2026}
J.~Lelis, C.~G.~Moreira, and E.~Silva,
On $C^k$-functions mapping $\Q$ into itself and Mahler's problem on Liouville numbers,
preprint, arXiv:2607.24427, 2026.

\bibitem{Lombardo2017}
D.~Lombardo,
On the analytic bijections of the rationals in $[0,1]$,
\textit{Rend. Lincei Mat. Appl.} \textbf{28} (2017), no.~1, 65--83.

\bibitem{Mahler1984}
K.~Mahler,
Some suggestions for further research,
\textit{Bull. Aust. Math. Soc.} \textbf{29} (1984), 101--108.

\bibitem{Maillet1906}
E.~Maillet,
\emph{Introduction \`a la th\'eorie des nombres transcendants et des propri\'et\'es arithm\'etiques des fonctions},
Gauthier-Villars, Paris, 1906.

\bibitem{MarquesMoreira2015}
D.~Marques and C.~G.~Moreira,
On a variant of a question proposed by K.~Mahler concerning Liouville numbers,
\textit{Bull. Aust. Math. Soc.} \textbf{91} (2015), 29--33.

\bibitem{MarquesMoreira2017}
D.~Marques and C.~G.~Moreira,
A positive answer for a question proposed by K.~Mahler,
\textit{Math. Ann.} \textbf{368} (2017), no.~3--4, 1059--1062.

\bibitem{MarquesRamirezSilva2016}
D.~Marques, J.~Ramirez, and E.~Silva,
A note on lacunary power series with rational coefficients,
\textit{Bull. Aust. Math. Soc.} \textbf{93} (2016), no.~3, 372--374.

\bibitem{Pila1991}
J.~Pila,
Geometric postulation of a smooth function and the number of rational points,
\textit{Duke Math. J.} \textbf{63} (1991), no.~2, 449--463.

\bibitem{Pila2009}
J.~Pila,
On the algebraic points of a definable set,
\textit{Selecta Math. (N.S.)} \textbf{15} (2009), no.~1, 151--170.

\bibitem{PilaRodriguezVillegas1999}
J.~Pila and F.~Rodriguez Villegas,
Concordant sequences and integral-valued entire functions,
\textit{Acta Arith.} \textbf{88} (1999), no.~3, 239--268.

\bibitem{PilaWilkie2006}
J.~Pila and A.~J.~Wilkie,
The rational points of a definable set,
\textit{Duke Math. J.} \textbf{133} (2006), no.~3, 591--616.

\bibitem{Stackel}
P.~St\"ackel,
Ueber arithmetische Eigenschaften analytischer Functionen,
\textit{Math. Ann.} \textbf{46} (1895), 513--520.

\bibitem{Surroca2002}
A.~Surroca,
Sur le nombre de points alg\'ebriques o\`u une fonction analytique transcendante prend des valeurs alg\'ebriques,
\textit{C. R. Math. Acad. Sci. Paris} \textbf{334} (2002), no.~9, 721--725.

\bibitem{Surroca2006}
A.~Surroca,
Valeurs alg\'ebriques de fonctions transcendantes,
\textit{Int. Math. Res. Not.} \textbf{2006} (2006), Art. ID 16834, 31 pp.

\bibitem{vandenDries1998}
L.~van den Dries,
\emph{Tame Topology and O-Minimal Structures},
London Mathematical Society Lecture Note Series, vol.~248,
Cambridge University Press, Cambridge, 1998.

\bibitem{vanderPoorten1968}
A.~J.~van der Poorten,
Transcendental entire functions mapping every algebraic number field into itself,
\textit{J. Austral. Math. Soc.} \textbf{8} (1968), 192--193.

\bibitem{Waldschmidt2003}
M.~Waldschmidt,
Algebraic values of analytic functions,
\textit{J. Comput. Appl. Math.} \textbf{160} (2003), no.~1--2, 323--333.

\bibitem{Waldschmidt2020}
M.~Waldschmidt,
Integer-valued functions, Hurwitz functions and related topics: a survey,
in \emph{Number Theory: Proceedings of the Journ\'ees Arithm\'etiques, 2019, XXXI},
De Gruyter Proceedings in Mathematics,
De Gruyter, Berlin, 2022, pp.~61--82.

\bibitem{Weierstrass1923}
K.~Weierstrass,
Briefe von K.~Weierstrass an L.~Koenigsberger,
\textit{Acta Math.} \textbf{39} (1923), no.~1, 226--239.

\end{thebibliography}
\end{document}